\documentclass[12pt]{article}
\usepackage{graphicx}
\usepackage{amsmath,amssymb,amsthm,amsfonts}
\newtheorem{thm}{Theorem}[section]
\newtheorem{cor}[thm]{Corollary}
\newtheorem{lem}[thm]{Lemma}

\newcommand{\Z}{{\mathbb{Z}}}

\newcommand{\1}{\partial}

\newcommand{\3}{\varepsilon}
\newcommand{\4}{\widetilde}

\begin{document}
\title{Rotational symmetry and properties of the\\ 
ancient solutions of Ricci flow on surfaces} 
\author{Shu-Yu Hsu\\
Department of Mathematics\\
National Chung Cheng University\\
168 University Road, Min-Hsiung\\
Chia-Yi 621, Taiwan, R.O.C.\\
e-mail: syhsu@math.ccu.edu.tw}
\date{March 11, 2010}
\smallbreak \maketitle
\begin{abstract}
We give a simple proof for the rotational symmetry of ancient 
solutions of Ricci flow on surfaces. As a consequence we obtain 
a simple proof of some results of P.~Daskalopoulos, R.~Hamilton and
N.~Sesum on the a priori estimates for the ancient 
solutions of Ricci flow on surfaces. We also give a simple proof
for the solution to be a Rosenau solution under some mild conditions 
on the solutions of Ricci flow on surfaces.
\end{abstract}

\vskip 0.2truein


\vskip 0.2truein
\setcounter{equation}{0}
\setcounter{section}{-1}

\setcounter{equation}{0}
\setcounter{thm}{0}

Recently there is a lot of study on Ricci flow on manifolds by R.~Hamilton
[7--11], S.Y.~Hsu \cite{Hs1}, \cite{Hs2}, G.~Perelman \cite{P1}, 
\cite{P2}, W.X.~Shi \cite{S1}, \cite{S2}, L.F.~Wu \cite{W1}, \cite{W2}, 
and others because it is an important tool in the study of geometry. 
Interested readers can read the book \cite{CK} by B.~Chow and 
D.~Knopf and the book \cite{MT} by J.~Morgan and G.~Tang for various topics 
on Ricci flow. One can also read the papers \cite{P1}, \cite{P2} of 
G.~Perelman for the most recent results on Ricci flow.

It is known \cite{P1} that the behavior of solutions of Ricci flow near 
a singularity can be described by the ancient solutions of Ricci flow. Hence
it is important to understand the ancient solutions of Ricci flow. 
Ancient solutions of Ricci flow on spheres and generalized Hopf fibrations was
studied by I.~Bakas, S.L.~Kong and L.~Ni \cite{BKN}. Ancient solutions of 
Ricci flow on noncompact surfaces was studied by S.C.~Chu \cite{Ch} and 
ancient solutions of Ricci flow on compact surfaces was studied by
P.~Daskalopoulos, R.S.~Hamilton, and N.~Sesum \cite{DHS}.

Let $g=(g_{ij})$ be an ancient solution of the Ricci flow,
\begin{equation*}
\frac{\1}{\1 t}g_{ij}=-2R_{ij}\quad\forall t<0,
\end{equation*}
on a compact surface which becomes singular at time $t=0$. It is observed by 
P.~Daskalopoulos, R.S.~Hamilton, and N.~Sesum in \cite{DHS} that by the 
results of \cite{C} and \cite{H3} the ancient solution of the Ricci flow 
can be parametrized by
\begin{equation*}
g(\cdot,t)=u(\cdot,t)ds_p^2
\end{equation*}
where 
$$
ds_p^2=d\psi^2+\cos^2\psi\,d\theta^2  
$$
is the spherical metric on the round sphere $S^2$ with coordinates 
$(\psi,\theta)$. Note that the Christoffel symbols for the spherical metric
are
$$
\Gamma_{12}^2=\Gamma_{21}^2=-\tan\psi,\quad\Gamma_{22}^1
=\frac{\sin2\psi}{2},\quad\Gamma_{22}^2=\Gamma_{11}^2=\Gamma_{11}^1
=\Gamma_{12}^1=0
$$
and
$$
\Delta_{S^2}f=f_{\psi\psi}-(\tan\psi )f_{\psi}
+(\sec^2\psi)f_{\theta\theta}
$$
for any function $f$ on the sphere and $u$ satisfies (\cite{DHS}),
\begin{equation}
u_t=\Delta_{S^2}\log u -2\quad\mbox{ in }S^2\times (-\infty,0).
\end{equation}
Let $v=u^{-1}$. Then $v$ satisfies
\begin{equation}
v_t=v\Delta_{S^2}v-|\nabla_{S^2}v|^2+2v^2\quad\mbox{ in }
S^2\times (-\infty,0).
\end{equation}
It was proved in \cite{DHS} that under a conformal 
change of $S^2$
\begin{equation*}
\lim_{t\to -\infty}v(\psi,\theta,t)=C_0\cos^2\psi
\quad\forall -\frac{\pi}{2}\le\psi\le\frac{\pi}{2}, 0\le\theta\le 2\pi
\end{equation*}
for some constant $C_0\ge 0$. Moreover $g$ is the contracting sphere 
with
$$
v(\psi,\theta,t)=\frac{1}{2(-t)}\quad\forall t<0,-\frac{\pi}{2}\le\psi
\le\frac{\pi}{2}, 0\le\theta\le 2\pi
$$ 
when $C_0=0$. When $C_0>0$, P.~Daskalopoulos, R.S.~Hamilton, and N.~Sesum 
(\cite{DHS}) proved that $g$ is the Rosenau solution \cite{R} with 
\begin{equation}
v(\psi,\theta,t)=-\mu\coth (2\mu t)+\mu\tanh (2\mu t)\sin^2\psi
\quad\forall t<0,-\frac{\pi}{2}\le\psi
\le\frac{\pi}{2}, 0\le\theta\le 2\pi
\end{equation}
for some $\mu>0$. An essential step in their proof is the proof of the
rotational symmetry of the ancient solutions of Ricci flow on surfaces. 
However their proof of the rotational symmetry is very hard and require 
the use of the difficult Lemma 2.10 of \cite{DHS}. In this paper under
a mild condition on $g$ we will give a simple proof of the rotational 
symmetry which avoids Lemma 2.10 of \cite{DHS}. 

As a consequence we also obtain simple proofs of some 
results of \cite{DHS} on the a priori estimates of the ancient solutions 
of Ricci flow on surfaces. Since the proof in \cite{DHS} that $g$
is the Rosenau solution when $C_0>0$ is hard, in this paper we will 
give a simple proof of this result under some mild conditions on $g$. 

Let $R(\cdot,t)$ be the scalar curvature of $g(\cdot,t)$. For any $z\in S^2$
and $\delta>0$, let $B_{\delta}(z)$ be the geodesic ball with center $z$
and radius $\delta$ on $S^2$ with respect to the round metric $ds_p^2$.
We first recall some results of \cite{DHS}.

\begin{lem}(Lemma 2.1 of \cite{DHS})  
For any $t_0<0$ there exists a constant $C_1>0$ such that
\begin{equation}
\sup_{S^2}\biggl(|\Delta_{S^2}v|+\frac{|\nabla_{S^2}v|^2}{v}\biggr)
\le C_1\quad\forall t\le t_0.
\end{equation}
\end{lem}

\begin{thm}(cf. Proposition 3.4 and Theorem 4.1 of \cite{DHS})
$v(\cdot,t)$ decreases and converges in $C^{1,\alpha}$ to some function 
$\4{v}\in C^{1,\alpha}(S^2)$ for any $0<\alpha<1$ as $t\to -\infty$.
Moreover under a conformal change of $S^2$, 
\begin{equation}
\4{v}(\psi,\theta)=C_0\cos^2\psi
\end{equation}
for some constant $C_0\ge 0$ and the convergence is uniform in 
$C^k(K)$ for any $k\in\Z^+$ and any compact set $K\subset 
S^2\setminus\{S,N\}$ where $S$ and $N$ are the south pole and the 
north pole of $S^2$.
\end{thm}
We will now assume that the coordinates on $S^2$ are chosen such that 
(5) holds for the rest of the paper. We will also assume that $C_0>0$ 
in (5) and there exist constants $0<a<3$, $t_0<0$, and $C>0$ such that 
\begin{equation}
v_{\theta}^2\le C v^{1+a}\cos^2\psi\quad\mbox{ on }S^2\times 
(-\infty,t_0]
\end{equation}
for the rest of the paper. Note that the Rosenau solution (3) satisfies
(6).

\begin{thm}(cf. Theorem 6.1 of \cite{DHS})
For any $|\psi|\le\pi/2$ and $t<0$, $u(\psi,\theta,t)$ is independent
of $\theta\in [0,2\pi]$. 
\end{thm}
\begin{proof}
A proof of this result without the condition (6) using Lemma 2.10 of 
\cite{DHS} is given in \cite{DHS}.
However the proof in \cite{DHS} is hard. Under the condition (6) we will 
give a different simple proof here using the technique of \cite{DK} and 
\cite{Hu}. Let $\theta_0\in (0,2\pi)$.
For any $q=(\theta,\psi)$, let $W(q)=(\theta+\theta_0,\psi)$ be 
the point on $S^2$ obtained by rotating $q$ an angle $\theta_0$ about the 
$z$-axis. For any $q\in S^2$, $t<0$, let $u_1(q,t)=u(W(q),t)$ and
$$
A(q,t)=\left\{\begin{aligned}
&\frac{\log u(q,t)-\log u_1(q,t)}{u(q,t)-u_1(q,t)}\quad\mbox{ if }
u(q,t)\ne u_1(q,t)\\
&\frac{1}{u(q,t)}\qquad\qquad\qquad\qquad\mbox{ if }
u(q,t)=u_1(q,t).
\end{aligned}\right.
$$
Let $t_2<t_1<t_0<0$. Then
$$
C_1\le A(q,t)\le C_2\quad\mbox{ on }S^2\times [t_2,t_1]
$$
for some constants $C_2>C_1>0$. For any $h\in C^{\infty}(S^2)$, 
$0\le h\le 1$, let $\eta$ be the solution of 
\begin{equation}
\left\{\begin{aligned}
\eta_t+A\Delta_{S^2}\eta=&0\qquad\mbox{ on }S^2\times [t_2,t_1)\\
\eta (q,t_1)=&h(q)\quad\mbox{ on }S^2.
\end{aligned}\right.
\end{equation}
By the maximum principle, $0\le\eta\le 1$ on $S^2\times [t_2,t_1)$. 
Since $u_1$ also satisfies (1), by (7),
\begin{align*}
&\int_{S^2}(u-u_1)(q,t_1)h(q)\,dV-\int_{S^2}(u-u_1)(q,t_2)
\eta (q,t_2)\,dV\nonumber\\
=&\int_{t_2}^{t_1}\frac{\1}{\1 t}\left(\int_{S^2}(u-u_1)\eta\,dV
\right)\,dt\nonumber\\
=&\int_{t_2}^{t_1}\int_{S^2}[(u-u_1)\eta_t+(u-u_1)_t\eta]\,dV\,dt
\nonumber\\
=&\int_{t_2}^{t_1}\int_{S^2}[(u-u_1)\eta_t
+\eta\Delta_{S^2}(\log u-\log u_1)]\,dV\,dt\nonumber\\
=&\int_{t_2}^{t_1}\int_{S^2}(u-u_1)[\eta_t+A\Delta_{S^2}\eta]
\,dV\,dt\nonumber\\
=&0.
\end{align*} 
Hence
\begin{equation}
\int_{S^2}(u-u_1)(q,t_1)h(q)\,dV=\int_{S^2}(u-u_1)(q,t_2)\eta (q,t_2)\,dV
\le\int_{S^2}(u-u_1)_+(q,t_2)\,dV
\end{equation}
We now choose a sequence of smooth functions $\{h_i\}_{i=1}^{\infty}$ on
$S^2$, $0\le h_i\le 1$ for all $i\in\Z^+$, such that $h_i$ converges
a.e. to the characteristic function of the set 
$\{q\in S^2:u(q,t_1)>u_1(q,t_1)\}$ as $i\to\infty$. Putting $h=h_i$ in
(8) and letting $i\to\infty$,
\begin{equation}
\int_{S^2}(u-u_1)_+(q,t_1)\,dV\le\int_{S^2}(u-u_1)_+(q,t_2)\,dV.
\end{equation}
Interchanging the role of $u$ and $u_1$ and repeating the above argument,
\begin{equation}
\int_{S^2}(u_1-u)_+(q,t_1)\,dV\le\int_{S^2}(u_1-u)_+(q,t_2)\,dV.
\end{equation}
By (9) and (10),
\begin{align}
\int_{S^2}|u-u_1|(q,t_1)\,dV\le&\int_{S^2}|u-u_1|(q,t_2)\,dV\nonumber\\
\le&\int_{-\frac{\pi}{2}}^{\frac{\pi}{2}}
\int_0^{2\pi}\left(\int_{\theta}^{\theta+\theta_0}|u_{\theta}|
(\rho,\psi,t_2)\,d\rho\right)\cos\psi\,d\theta\,d\psi.
\end{align}
By (6),
\begin{equation}
v^{3-a}u_{\theta}^2\le C\cos^2\psi\quad\mbox{ on }S^2
\quad\forall t\le t_0.
\end{equation}
By Theorem 2,
\begin{equation}
C_0\cos^2\psi\le v\le\max_{S^2}v(\cdot,t_0)\quad\mbox{ on }S^2
\quad\forall t\le t_0.
\end{equation}
By (12) and (13),
\begin{equation}
|u_{\theta}|\le C'(\cos\psi)^{a-2}\quad\mbox{ on }S^2\quad\forall t\le t_0.
\end{equation}
By Theorem 2, $u(\cdot,t)$ converges uniformly to $1/(C_0\cos^2\psi)$ 
in $C^{1,\alpha}$ for any $0<\alpha<1$ on any compact set $K\subset 
S^2\setminus\{S,N\}$ as $t\to -\infty$. Hence $u_{\theta}(\cdot,t)$ 
converges uniformly to $0$ on any compact set $K\subset S^2
\setminus\{S,N\}$ 
as $t\to -\infty$. Letting $t_2\to -\infty$ in (11), by (14) 
and the Lebesgue dominated convergence theorem,
\begin{align*}
&\int_{S^2}|u-u_1|(q,t_1)\,dV\le 0\quad\forall t_1<0\nonumber\\
\Rightarrow\quad&u(\theta,\psi,t)\equiv u(\theta+\theta_0,\psi,t)
\quad\forall\theta,\theta_0\in [0,2\pi],|\psi|\le\pi/2,t<0.
\end{align*}
Thus $u(\theta,\psi,t)$ is independent of $\theta$ and the 
proposition follows.
\end{proof}

We can now write
$$
u(\psi,t)=u(\theta,\psi,t)\quad\mbox{ and }\quad
v(\psi,t)=v(\theta,\psi,t)\quad\forall\theta\in [0,2\pi],|\psi|
\le\pi/2,t<0
$$
and let
\begin{equation*}
f=\Delta_{S^2}v.
\end{equation*}
Then $f\in C^{\infty}(S^2\times (-\infty,0))$ and by Lemma 1 there 
exists a constant $C_1>0$ such that
\begin{equation}
\sup_{S^2}|f(\cdot,t)|\le C_1\quad\forall t\le t_0.
\end{equation} 

\begin{cor}(Propositon 2.5 and Lemma 2.7 of \cite{DHS})
For any $t_0<0$ there exists a constant $C_2>0$ such that
\begin{equation}
|v_{\psi\psi}|+|(\sec\psi)v_{\psi}|\le C_2\quad\mbox{ on }S^2
\quad\forall t\le t_0.
\end{equation}
\end{cor}
\begin{proof}
A proof of this result is given in \cite{DHS}. However the proof 
in \cite{DHS} is hard. We will use ODE technique to give 
a simple proof here. Let $t_0<0$. Since $v$ is independent of $\theta$, 
\begin{equation}
v_{\psi\psi}-(\tan\psi)v_{\psi}=f\quad\forall |\psi|<\pi/2,t<0.
\end{equation}
Since $v((\pi/2)+\delta,t)=v((\pi/2)-\delta,t)$ and 
$v(-(\pi/2)+\delta,t)=v(-(\pi/2)-\delta,t)$ for any 
$0<\delta<\pi/2$ and $t<0$,
\begin{equation}
v_{\psi}(\pi/2,t)=v_{\psi}(-\pi/2,t)=0\quad\forall t<0.
\end{equation} 
By (15), (17), (18) and the mean value theorem,
\begin{align}
((\cos\psi)v_{\psi})_{\psi}
=&\cos\psi(v_{\psi\psi}-(\tan\psi)v_{\psi})=(\cos\psi)f\nonumber\\
\Rightarrow\qquad\qquad\quad v_{\psi}=&\frac{1}{\cos\psi}
\int_{\frac{\pi}{2}}^{\psi}(\cos\rho)f(\rho,t)\,d\rho
\quad\forall -\pi/2<\psi<\pi/2,t<0\\
=&\frac{1}{\cos\psi}\int_{\frac{\pi}{2}}^{\psi}
(\cos\rho-\cos(\pi/2))f(\rho,t)\,d\rho\nonumber\\
\Rightarrow\qquad\qquad\,\, |v_{\psi}|\le&\frac{C_1}{\cos\psi}
\int_{\frac{\pi}{2}}^{\psi}|\rho-(\pi/2)|\,d\rho\nonumber\\
\le&\frac{C_1}{2\cos\psi}((\pi/2)-\psi)^2\nonumber\\
\Rightarrow\quad\,\,\,|(\sec\psi)v_{\psi}|
\le&C_1\frac{((\pi/2)-\psi)^2}{2\cos^2\psi}\le C'
\quad\forall -\pi/4\le\psi<\pi/2,t\le t_0.
\end{align}
By (15), (17) and (20),
\begin{equation}
|v_{\psi\psi}|\le |\sec\psi||v_{\psi}|+|f|\le C_1+C'
\quad\forall -\pi/4\le\psi<\pi/2,t\le t_0.
\end{equation}
By (17) and (18),
\begin{equation}
v_{\psi}=\frac{1}{\cos\psi}\int_{-\frac{\pi}{2}}^{\psi}
(\cos\rho)f(\rho,t)\,d\rho\quad\forall -\pi/2<\psi<\pi/2,t<0.
\end{equation}
By (22) and an argument as before,
\begin{equation}
|v_{\psi\psi}|+|(\sec\psi)v_{\psi}|\le C\quad\forall -\pi/2<\psi\le\pi/4,
t\le t_0.
\end{equation}
By (20), (21) and (23), we get (16) and the corollary follows.
\end{proof}

\begin{cor}
Let $t_0<0$. Then there exists a constant $C_3>0$ such that
\begin{equation}
|(\cos\psi)v_{\psi\psi\psi}(\psi,t)|\le C_3\quad\forall
|\psi|<\frac{\pi}{2}, t\le t_0.
\end{equation}
\end{cor}
\begin{proof}
By the proof of Corollary 4, (19) and (22) holds. Differentiating 
(19) with respect to $\psi$,
\begin{equation}
v_{\psi\psi}(\psi,t)=\frac{\sin\psi}{\cos^2\psi}
\int_{\frac{\pi}{2}}^{\psi}(\cos\rho)f(\rho,t)\,d\rho+f(\psi,t)
\quad\forall |\psi|<\frac{\pi}{2},t<0.
\end{equation}
Now by (2) (cf. \cite{DHS}),
\begin{equation*}
R=\frac{v_t}{v}=\Delta_{S^2}v-\frac{|\nabla_{S^2}v|^2}{v}+2v
\quad\forall |\psi|<\frac{\pi}{2},t<0.
\end{equation*}
Hence
\begin{equation}
f=\Delta_{S^2}v=R+\frac{|\nabla_{S^2}v|^2}{v}-2v
\quad\forall |\psi|<\frac{\pi}{2},t<0.
\end{equation}
By (25) and (26),
\begin{equation}
v_{\psi\psi}(\psi,t)=\frac{\sin\psi}{\cos^2\psi}
\int_{\frac{\pi}{2}}^{\psi}(\cos\rho)f(\rho,t)\,d\rho
+R(\psi,t)+\frac{v_{\psi}^2}{v}-2v
\quad\forall |\psi|<\frac{\pi}{2},t<0.
\end{equation}
Differentiating (27) with respect to $\psi$,
\begin{align*}
v_{\psi\psi\psi}=&\frac{1}{\cos\psi}
\int_{\frac{\pi}{2}}^{\psi}(\cos\rho)f(\rho,t)\,d\rho
+2\frac{\sin^2\psi}{\cos^3\psi}
\int_{\frac{\pi}{2}}^{\psi}(\cos\rho)f(\rho,t)\,d\rho
+(\tan\psi)f+R_{\psi}\\
&\qquad+2\frac{v_{\psi}v_{\psi\psi}}{v}
-\frac{v_{\psi}^3}{v^2}-2v_{\psi}
\quad\forall |\psi|<\frac{\pi}{2},t<0.
\end{align*}
Hence
\begin{align}
(\cos\psi)v_{\psi\psi\psi}=&\int_{\frac{\pi}{2}}^{\psi}
(\cos\rho)f(\rho,t)\,d\rho+2\frac{\sin^2\psi}{\cos^2\psi}
\int_{\frac{\pi}{2}}^{\psi}(\cos\rho)f(\rho,t)\,d\rho
+(\sin\psi)f\nonumber\\
&\qquad+R_{\psi}\cos\psi+2\frac{\cos\psi}{\sqrt{v}}
\cdot\frac{v_{\psi}}{\sqrt{v}}
v_{\psi\psi}-\frac{\cos\psi}{\sqrt{v}}\cdot
\left(\frac{v_{\psi}}{\sqrt{v}}\right)^3\nonumber\\
&\qquad -2\frac{v_{\psi}}{\sqrt{v}}\cdot(\sqrt{v}\cos\psi)
\end{align}
holds for any $|\psi|<\frac{\pi}{2}$ and $t<0$.
Now by the Harnack inequality \cite{H4} and the maximum 
principle, $R_t\ge 0$ on $S^2\times (-\infty,0)$ and
\begin{equation}
0<R\le\max_{S^2}R(\cdot,t_0)\quad\mbox{ on } S^2\times 
(-\infty,t_0].
\end{equation}
By (13), (29), and Shi's derivative estimates \cite{H5},
\begin{align}
&(C_0\cos^2\psi )R_{\psi}^2
\le vR_{\psi}^2=|\nabla_gR(z,t)|\le\frac{C}{\sqrt{|t|}}\quad\mbox{ on }
S^2\times (-\infty,t_0]\nonumber\\
\Rightarrow\quad&|(\cos\psi)R_{\psi}|\le\frac{C'}{|t_0|^{\frac{1}{4}}}
\quad\mbox{ on }S^2\times (-\infty,t_0].
\end{align}
By (13), Lemma 1, and Corollary 4 there exists a constant $C>0$ such that
\begin{equation}
\left|\frac{\cos\psi}{\sqrt{v}}\cdot\frac{v_{\psi}}{\sqrt{v}}
v_{\psi\psi}\right|+\left|\frac{\cos\psi}{\sqrt{v}}\cdot
\left(\frac{v_{\psi}}{\sqrt{v}}\right)^3\right|+\left|
\frac{v_{\psi}}{\sqrt{v}}\cdot(\sqrt{v}\cos\psi)\right|
\le C\quad\mbox{ on }S^2\times (-\infty,t_0].
\end{equation}
Now by (15),
\begin{align}
\left|\int_{\frac{\pi}{2}}^{\psi}(\cos\rho)f(\rho,t)\,d\rho\right|
=&\left|\int_{\frac{\pi}{2}}^{\psi}(\cos\rho-\cos(\pi/2))
f(\rho,t)\,d\rho\right|\nonumber\\
\le&C_1\int_{\psi}^{\frac{\pi}{2}}|\rho-(\pi/2)|\,d\rho\nonumber\\
\le&C_1((\pi/2)-\psi)^2\nonumber\\
\le&C'\cos^2\psi\quad\forall -\pi/4\le\psi<\pi/2,t\le t_0.
\end{align}
Hence
\begin{equation}
\frac{\sin^2\psi}{\cos^2\psi}
\left|\int_{\frac{\pi}{2}}^{\psi}(\cos\rho)f(\rho,t)\,d\rho\right|
\le C''\quad\forall -\pi/4\le\psi<\pi/2,t\le t_0.
\end{equation}
By (15), (28), (30), (31), (32) and (33), there exists a constant $C>0$
such that
\begin{equation}
|(\cos\psi)v_{\psi\psi\psi}|\le C\quad\forall -\pi/4\le\psi<\pi/2,
t\le t_0.
\end{equation}
Similarly by using (22) and repeating the above argument,
\begin{equation}
|(\cos\psi)v_{\psi\psi\psi}|\le C\quad\forall -\pi/2<\psi\le\pi/4,
t\le t_0.
\end{equation}
By (34) and (35) we get (24) and the lemma follows.
\end{proof}

As in \cite{DHS} we introduce the Mercator's projection of 
the sphere $S^2$ onto the cylinder with coordinates $(x,\theta)$ 
which is given by
\begin{equation*}
\cosh x=\sec\psi\quad\mbox{ and }\quad\sinh x=\tan\psi
\end{equation*}
and
\begin{equation*}
\frac{d x}{d\psi}=\sec\psi.
\end{equation*}
Then in the cylindrical coordinates the metric $g$ can be written as
$$
g(x,\theta,t)=U(x,\theta,t)ds^2,\quad ds^2=dx^2+d\theta^2
$$
for some function $U$ that satisfies
\begin{equation*}
U(x,\theta,t)=u(\psi,t)\cos^2\psi.
\end{equation*}
Let $w=U^{-1}$. Then 
\begin{equation*}
v(\psi,t)=w(x,t)\cos^2\psi
\end{equation*}
where $\psi$ and $x$ are related by the Mercator's projection.
Let 
$$
Q(x,t)=w_{xx}(x,t)-4w(x,t)
$$ 
and
$$
F(x,t)=Q_x^2(x,t).
$$
Let $H(\psi,t)$ be the function $F$ in Mercator's coordinates on $S^2$.

\begin{cor}
For any $t_0<0$ there exists a constant $C>0$ such that
\begin{equation*}
H(\psi,t)\le C\quad\forall |\psi|<\frac{\pi}{2}, t\le t_0.
\end{equation*} 
\end{cor}
\begin{proof}
This result is proved in \cite{DHS} using Proposition 2.5 and 
Corollary 2.12 of \cite{DHS} whose proof is hard. We will
give a simple proof here. By direct computation,
\begin{align}
Q_x=&[[(v\sec^2\psi)_{\psi}\cos\psi]_{\psi}\cos\psi]_{\psi}\cos\psi
-4(v\sec^2\psi)_{\psi}\cos\psi\nonumber\\
=&-2(\cos\psi)v_{\psi}+3((\sec\psi)v_{\psi}+(\sin\psi)v_{\psi\psi})
+(\cos\psi)v_{\psi\psi\psi}.
\end{align}
By (36), Corollary 4, and Corollary 5 the corollary follows.
\end{proof}

\begin{lem}
For any $t_2<t_0<0$ and $\3>0$, there exists $\psi_0\in (0,\pi/2)$
such that
\begin{equation}
H(\psi,t)\le\3\quad\forall \psi_0\le|\psi|<\frac{\pi}{2}, 
t_2\le t\le t_0.
\end{equation} 
\end{lem}
\begin{proof}
By Corollary 4,
\begin{equation}
|(\cos\psi)v_{\psi}|\le C_2\cos^2\psi\quad\forall |\psi|
<\frac{\pi}{2},t\le t_0.
\end{equation}
Since $v_t=Rv\ge 0$ in $S^2\times (-\infty,0)$,
\begin{equation}
0<\min_{S^2}v(\cdot,2t_2)\le v(z,t)\le\max_{S^2}v(\cdot,t_0)
\quad\mbox{ on }S^2\times [2t_2,t_0].
\end{equation}
By (39) the equation (2) for $v$ is uniformly parabolic on $S^2\times 
[2t_2,t_0]$. By the parabolic Schauder estimates \cite{LSU} 
and the compactness of $S^2\times [t_2,t_0]$, there exists a constant
$C>0$ such that
\begin{equation}
|v_{\psi}|+|v_{\psi\psi}|+|v_{\psi\psi\psi}|+|v_{\psi\psi\psi\psi}|
\le C\quad\forall |\psi|<\frac{\pi}{2},t_2\le t\le t_0
\end{equation}
and
\begin{equation}
|f(\rho,t)-f(\psi,t)|=|\Delta_{S^2}v(\rho,t)-\Delta_{S^2}v(\psi,t)|
\le C|\rho-\psi|^{\alpha}
\end{equation}
with $\alpha=1$ for any $\rho,\psi\in (-\pi/2,\pi/2)$ and 
$t_2\le t\le t_0$. Hence
\begin{equation}
|(\cos\psi)v_{\psi\psi\psi}|\le C\cos\psi\quad\forall 
|\psi|<\frac{\pi}{2},t_2\le t\le t_0.
\end{equation}
Now by Corollary 4,
\begin{align}
|(\sec\psi)v_{\psi}+(\sin\psi)v_{\psi\psi}|
\le&|(\sec\psi)v_{\psi}+v_{\psi\psi}|+|1-\sin\psi||v_{\psi\psi}|
\nonumber\\
\le&|(\sec\psi)v_{\psi}+v_{\psi\psi}|+C_2|1-\sin\psi|
\quad\mbox{ on }S^2\quad\forall t\le t_0.
\end{align}
By (17) and Corollary 4,
\begin{align}
|(\sec\psi)v_{\psi}+v_{\psi\psi}|
=&|(\sec\psi)v_{\psi}+(\tan\psi)v_{\psi}+f|\nonumber\\
\le&|2(\sec\psi)v_{\psi}+f|+|1-\sin\psi||(\sec\psi)v_{\psi}|
\nonumber\\
\le&|2(\sec\psi)v_{\psi}+f|+C_2|1-\sin\psi|
\quad\mbox{ on }S^2\quad\forall t\le t_0.
\end{align}
By (19),
\begin{equation}
|2(\sec\psi)v_{\psi}(\psi,t)+f(\psi,t)|
=\left|\frac{2}{\cos^2\psi}\int_{\frac{\pi}{2}}^{\psi}(\cos\rho)
f(\rho,t)\,d\rho+f(\psi,t)\right|.
\end{equation}
By the mean value theorem for any $\rho\in (-\pi/4,\pi/2)$, there
exists a constant $\phi_{\rho}\in (\rho,\pi/2)$ such that
\begin{equation}
\cos\rho =\cos\rho -\cos(\pi/2)=((\pi/2)-\rho)\sin\phi_{\rho}.
\end{equation}
By (15), (45) and (46),
\begin{align}
|2(\sec\psi)v_{\psi}(\psi,t)+f(\psi,t)|
=&\left|\frac{2}{\cos^2\psi}\int_{\frac{\pi}{2}}^{\psi}
((\pi/2)-\rho)f(\rho,t)\sin\phi_{\rho}\,d\rho+f(\psi,t)\right|
\nonumber\\
\le&I_1+I_2+I_3
\end{align}
where
\begin{align}
I_1=&\frac{2}{\cos^2\psi}\left|\int_{\frac{\pi}{2}}^{\psi}
((\pi/2)-\rho)(\sin\phi_{\rho}-\sin(\pi/2))f(\rho,t)\,d\rho
\right|\nonumber\\
\le&\frac{2}{\cos^2\psi}\int_{\psi}^{\frac{\pi}{2}}
|(\pi/2)-\rho||\sin\phi_{\rho}-\sin(\pi/2)||f(\rho,t)|\,d\rho
\nonumber\\
\le&\frac{2}{\cos^2\psi}\int_{\psi}^{\frac{\pi}{2}}
|(\pi/2)-\rho||\phi_{\rho}-(\pi/2)||f(\rho,t)|\,d\rho\nonumber\\
\le&\frac{2}{\cos^2\psi}\int_{\psi}^{\frac{\pi}{2}}
|(\pi/2)-\rho|^2|f(\rho,t)|\,d\rho\nonumber\\
\le&2C_1\frac{(\frac{\pi}{2}-\psi)^3}{3\cos^2\psi}\nonumber\\
\le&C|(\pi/2)-\psi|\qquad\qquad
\forall\psi\in (-\pi/4,\pi/2),t\le t_0,
\end{align}
\begin{equation}
I_2=\frac{2}{\cos^2\psi}\int_{\psi}^{\frac{\pi}{2}}
((\pi/2)-\rho)|f(\rho,t)-f(\psi,t)|\,d\rho
\quad\forall\psi\in (-\pi/4,\pi/2),t\le t_0,
\end{equation}
and
\begin{equation}
I_3=\left|1-\frac{(\frac{\pi}{2}-\psi)^2}{\cos^2\psi}\right|
|f(\psi,t)|
\le C_1\left|1-\frac{1}{\sin^2\psi'}\right|
\le C\cos^2\psi'
\end{equation}
for some $\psi'\in (\psi,\pi/2)$ where $\psi\in (\pi/4,\pi/2)$
and $t\le t_0$. By (41) and (49), 
\begin{align}
I_2\le&\frac{2}{\cos^2\psi}\int_{\psi}^{\frac{\pi}{2}}
((\pi/2)-\rho)(\rho-\psi)^{\alpha}\,d\rho\nonumber\\
\le&\frac{2((\pi/2)-\psi)^{\alpha}}{\cos^2\psi}
\int_{\psi}^{\frac{\pi}{2}}((\pi/2)-\rho)\,d\rho\nonumber\\
\le&\frac{((\pi/2)-\psi)^{2+\alpha}}{\cos^2\psi}\nonumber\\
\le&C((\pi/2)-\psi)^{\alpha}
\quad\forall\psi\in (-\pi/4,\pi/2),t\le t_0,
\end{align}
with $\alpha=1$.
By (36), (38), (42), (43), (44), (47), (48), (50) and 
(51), there exists a constant $\psi_1\in (\pi/4,\pi/2)$ such that
\begin{equation}
H(\psi,t)<\3\quad\forall\psi_1\le\psi<\pi/2,t_2\le t\le t_0.
\end{equation} 
Similarly there exists a constant $\psi_2\in (\pi/4,\pi/2)$ 
such that
\begin{equation}
H(\psi,t)<\3\quad\forall -\pi/2<\psi\le -\psi_2,t_2\le t\le t_0.
\end{equation} 
Let $\psi_0=\max (\psi_1,\psi_2)$. By (52) and (53), we get (37)
and the lemma follows.
\end{proof}

\begin{lem}
For any $t_2<t_0<0$, there exists a constant $C>0$ such that
\begin{equation}
|H_{\psi}(\psi,t)|\le C\quad\forall |\psi|<\frac{\pi}{2},
t_2\le t\le t_0.
\end{equation}
\end{lem}
\begin{proof}
Differentiating (36) with respect to $\psi$,
\begin{align}
Q_{x\psi}=&2(\sin\psi)v_{\psi}-2(\cos\psi)v_{\psi\psi}
+3(\sec\psi\tan\psi\, v_{\psi}+(\sec\psi )v_{\psi\psi}
+(\cos\psi)v_{\psi\psi})\nonumber\\
&\qquad+2(\sin\psi)v_{\psi\psi\psi}
+(\cos\psi)v_{\psi\psi\psi\psi}.
\end{align}
Now
\begin{equation}
\sec\psi\tan\psi\, v_{\psi}+(\sec\psi )v_{\psi\psi}
=((\sec\psi)v_{\psi}+v_{\psi\psi})\tan\psi
+\frac{\cos\psi}{1+\sin\psi}v_{\psi\psi}.
\end{equation}
By the proof of Lemma 7 and the mean value theorem,
\begin{align}
|(\sec\psi)v_{\psi}+v_{\psi\psi}|\le&C(|1-\sin\psi|
+|(\pi/2)-\psi|+\cos^2\psi)\nonumber\\
\le&C'(|(\pi/2)-\psi|+\cos^2\psi)
\end{align}
holds for any $\pi/4\le\psi<\pi/2$, $t_2\le t\le t_0$.
By (56), (57), Corollary 4 and the mean value theorem,
there exists a constant $C>0$ such that
\begin{equation}
|\sec\psi\tan\psi\, v_{\psi}+(\sec\psi )v_{\psi\psi}|
\le C\quad\forall -\frac{\pi}{4}\le\psi<\frac{\pi}{2},
t_2\le t\le t_0.
\end{equation}
By (40), (55) and (58),
\begin{equation}
|Q_{x\psi}|\le C\quad\forall -\frac{\pi}{4}\le\psi<\frac{\pi}{2},
t_2\le t\le t_0.
\end{equation}
Similarly
\begin{equation}
|Q_{x\psi}|\le C\quad\forall -\frac{\pi}{2}<\psi\le\frac{\pi}{4},
t_2\le t\le t_0.
\end{equation}
By (59) and (60),
\begin{equation}
|Q_{x\psi}(\psi,t)|\le C\quad\forall |\psi|<\frac{\pi}{2},
t_2\le t\le t_0.
\end{equation}
Since $H_{\psi}=2Q_xQ_{x\psi}$, by (61) and Corollary 6
the lemma follows.
\end{proof}

\begin{thm}
Let $C_0>0$. Suupose there exist constants $0<\alpha<1$,
$t_0<0$, $0<a<3$, and $C>0$ such that 
\begin{equation}
(\cos\psi)^{1-\alpha}|v_{\psi\psi\psi}|\le C\quad\mbox{ on }
S^2\times (-\infty,t_0]
\end{equation} 
and (6) and (41) hold for any $\rho,\psi\in (-\pi/2,\pi/2)$ 
and $t\le t_0$. Then $v(\psi,t)=v(\theta,\psi,t)$ satisfies 
(3) for some constant $\mu>0$.
\end{thm}
\begin{proof}
As observed in \cite{DHS} it suffices to show that 
\begin{equation}
H(\psi,t)\equiv 0\quad\forall |\psi|<\pi/2,t<0.
\end{equation}
A proof of Theorem 9 without the assumptions (6), (41), and (62) 
was given in section 7 of \cite{DHS}. However the proof in 
\cite{DHS} is hard and requires
the use of the Harnack inequality and Shi's derivative estimates for 
Ricci flow. In this paper we find that under the mild assumptions 
(6), (41), and (62) we can prove the theorem by a simple argument. 
Let $t_0<0$. As observed in \cite{DHS} 
$H(\psi,t)$ satisfies
\begin{equation}
\frac{\1 H}{\1 t}\le v\Delta_{S^2}H\quad\mbox{ on }(S^2\setminus
\{S,N\})\times (-\infty,0).
\end{equation}
Since $v_t=Rv>0$ in $S^2\times (-\infty,0)$, by (64),
\begin{equation}
\frac{\1}{\1 t}\left(\frac{H}{v}\right)=\frac{H_t}{v}-\frac{H}{v^2}v_t
\le\Delta_{S^2}H\quad\mbox{ on }(S^2\setminus
\{S,N\})\times (-\infty,0).
\end{equation}
Let $0<\delta<1$. Then by (65) and Lemma 8,
\begin{align*}
&\int_{S^2\setminus (B_{\delta}(S)\cup B_{\delta}(N))}\frac{H}{v}(\psi,t_1)
\,dV
-\int_{S^2\setminus (B_{\delta}(S)\cup B_{\delta}(N))}\frac{H}{v}(\psi,t_2)
\,dV\\
=&\int_{t_2}^{t_1}\int_{S^2\setminus (B_{\delta}(S)\cup B_{\delta}(N))}
\frac{\1}{\1 t}\left(\frac{H}{v}\right)\,dV\,dt\\
\le&\int_{t_2}^{t_1}\int_{S^2\setminus (B_{\delta}(S)\cup B_{\delta}(N))}
\Delta_{S^2}H\,dV\,dt\\
=&\int_{t_2}^{t_1}\int_{\1 B_{\delta}(S)\cup \1 B_{\delta}(N)}
\frac{\1 H}{\1 n}\,d\sigma\,dt\\
\to& 0\quad\mbox{ as }\delta\to 0\qquad\qquad\qquad\forall t_2<t_1<0.
\end{align*}
Hence
\begin{equation}
\int_{S^2}\frac{H}{v}(\psi,t_1)\,dV
\le\int_{S^2}\frac{H}{v}(\psi,t_2)\,dV
=2\pi\int_{-\frac{\pi}{2}}^{\frac{\pi}{2}}\frac{H}{v}(\psi,t_2)\cos\psi
\,d\psi\quad\forall t_2<t_1<0.
\end{equation}
By (36),
\begin{equation}
H\le 18[(\cos^2\psi)v_{\psi}^2
+((\sec\psi)v_{\psi}+(\sin\psi)v_{\psi\psi})^2
+(\cos^2\psi)v_{\psi\psi\psi}^2].
\end{equation}
By the proof of Lemma 7 and the mean value theorem,
\begin{align}
|(\sec\psi)v_{\psi}+v_{\psi\psi}|\le&C(|1-\sin\psi|
+|(\pi/2)-\psi|^{\alpha}+\cos^2\psi)\nonumber\\
\le&C'(((\pi/2)-\psi)^{\alpha}+\cos^2\psi)
\end{align}
holds for any $\pi/4\le\psi<\psi/2$, $t\le t_0$. By (13), (38), 
(62), (67) and (68),
\begin{align}
&H\le C(\cos^4\psi+(((\pi/2)-\psi)^{\alpha}+\cos^2\psi)^2
+\cos^{2\alpha}\psi)\quad\forall \pi/4\le\psi<\pi/2,
t\le t_0\nonumber\\
\Rightarrow\quad&\frac{H}{v}(\psi,t)\cos\psi
\le C\cos^{2\alpha-1}\psi
\le C'((\pi/2)-\psi)^{2\alpha-1}\quad\forall \pi/4\le\psi
<\pi/2, t\le t_0.
\end{align}
Similarly
\begin{equation}
\frac{H}{v}(\psi,t)\cos\psi\le C|\psi+(\pi/2)|^{2\alpha-1}
\quad\forall -\pi/2<\psi\le -\pi/4,t\le t_0.
\end{equation}
By (13), (38), (43), (62), Corollary 4 and Corollary 5,
\begin{equation}
\frac{H}{v}(\psi,t)\cos\psi\le C\quad\forall |\psi|\le\pi/4,
t\le t_0.
\end{equation}
By (69), (70), (71) and the Lebesgue dominated convergence theorem,
\begin{equation}
\lim_{t_2\to -\infty}\int_{-\frac{\pi}{2}}^{\frac{\pi}{2}}
\frac{H}{v}(\psi,t_2)\cos\psi\,d\psi=0.
\end{equation}
Hence letting $t_2\to -\infty$ in (66), by (72),
\begin{align*}
&\int_{-\frac{\pi}{2}}^{\frac{\pi}{2}}\frac{H}{v}(\psi,t_1)
\cos\psi\,d\psi=0\quad\forall t_1<0\\
\Rightarrow\quad&H\equiv 0\quad\mbox{ in }S^2\times (-\infty,0)
\end{align*}
and the theorem follows.
\end{proof}


\begin{thebibliography}{99}

\bibitem{BKN} I.~Bakas, S.L.~Kong and L.~Ni, {\em Ancient solutions 
of Ricci flow on spheres and generalized Hopf fibrations},
http://arxiv.org/abs/0906.0589.

\bibitem{C} B.~Chow, {\em The Ricci flow on the $2$-sphere}, J. Diff. 
Geom. {\bf 33} (1991), 325--334.

\bibitem{CK} B.~Chow and D.~Knopf, {\em The Ricci flow: An introduction},
Mathematical Surveys and Monographs, Volume 110, Amer. Math. Soc.,
Providence, R.I., U.S.A. 2004.

\bibitem{Ch} S.C.~Chu, {\em Type II ancient solutions to the Ricci flow
on surfaces}, Comm. Anal. Geom. {\bf 15} (2007), no. 1, 195--216.

\bibitem{DK} B.E.J.~Dahlberg and C.~Kenig, {\em Non-negative
solutions of generalized porous medium equations}, Revista 
Matem\'atica Iberoamericana {\bf 2} (1986), 267--305.

\bibitem{DHS} P.~Daskalopoulos, R.S.~Hamilton, N.~Sesum, {\em 
Classification of compact ancient ancient solutions to the Ricci flow
on surfaces}, http://arxiv.org/abs/0902.1158.

\bibitem{H1} R.S.~Hamilton, {\em Three-manifolds with positive Ricci 
curvature}, J. Diff. Geom. {\bf 17} (1982), no. 2, 255--306.

\bibitem{H2} R.S.~Hamilton, {\em Four-manifolds with positive curvature
operator}, J. Diff. Geom. {\bf 24} (1986), no. 2, 153--179.

\bibitem{H3} R.S.~Hamilton, {\em The Ricci flow on surfaces},
Contemp. Math. {\bf 71} (1988), 237--261.

\bibitem{H4} R.S.~Hamilton, {\em The Harnack estimate for the Ricci flow},
J. Diff. Geom. {\bf 37} (1993), no. 1, 225--243.

\bibitem{H5} R.S.~Hamilton, {\em The formation of singularities in the 
Ricci flow}, Surveys in differential geometry, Vol. II (Cambridge, MA, 1993),
7--136, International Press, Cambridge, MA, 1995.

\bibitem{Hs1} S.Y.~Hsu, {\em A simple proof on the non-existence of 
shrinking breathers for the Ricci flow}, Calculus of Variations
and P.D.E. {\bf 27} (2006), no. 1, 59--73.

\bibitem{Hs2} S.Y.~Hsu, {\em Generalized $\mathcal{L}$-geodesic and 
monotonicity of the generalized reduced volume in the Ricci flow},
J. Math. Kyoto Univ. {\bf 49} (2009), no. 3, 503--571. 

\bibitem{Hu} K.M.~Hui, {\em Existence of solutions of the equation 
$u_t=\Delta\log u$}, Nonlinear Analysis, TMA {\bf 37} (1999),
no. 7, 875--914. 

\bibitem{LSU} O.A.~Ladyzenskaya, V.A.~Solonnikov, and
N.N.~Uraltceva, {\em Linear and quasilinear equations of
parabolic type}, Transl. Math. Mono. Vol 23,
Amer. Math. Soc., Providence, R.I., U.S.A. 1968.

\bibitem{MT} J.~Morgan and G.~Tang, {\em Ricci flow and the Poincar\'e
conjecture}, Clay Mathematics Monographs Volume 3, American Mathematical 
Society, Providence, RI, USA, 2007.

\bibitem{P1} G.~Perelman, {\em The entropy formula for the Ricci flow and its 
geometric applications}, http://arxiv.org/abs/math/0211159.

\bibitem{P2} G.~Perelman, {\em Ricci flow with surgery on three-manifolds},
http://arxiv.org/abs\linebreak 
/math/0303109.

\bibitem{R} P.~Rosenau, {\em Fast and superfast diffusion processes}, 
Phys. Rev. Letter {\bf 74} (1995), 1056--1059.

\bibitem{S1} W.X.~Shi, {\em Deforming the metric on complete Riemannian 
manifolds}, J. Differential Geom. {\bf 30} (1989), 223--301.

\bibitem{S2} W.X.~Shi, {\em Ricci deformation of the metric on complete 
non-compact Riemannian manifolds}, J. Differential Geom. {\bf 30} (1989),
303--394.

\bibitem{W1} L.F.~Wu, {\em The Ricci flow on complete $R^2$},
Comm. in Analysis and Geometry {\bf 1} (1993), 439--472.

\bibitem{W2} L.F.~Wu, {\em A new result for the porous
medium equation}, Bull. Amer. Math. Soc. {\bf 28} (1993), 90--94.

\end{thebibliography}
\end{document}